\documentclass[10pt]{amsart}
\usepackage{amsmath,amssymb,amsthm,graphicx,a4wide}

\newtheorem{theorem}{Theorem}    % Standard theorem environment
\newtheorem{lemma}{Lemma} 
 
\newtheorem{proposition}{Proposition} 
 
\newtheorem{corollary}{Corollary}
\theoremstyle{definition}
\newtheorem{definition}[theorem]{Definition}

\newcommand{\Z}{\mathbb{Z}}

\title[Kernel of braid representation and knot polynomials]{A kernel of a braid group representation yields a knot with trivial knot polynomials}
\author{Tetsuya Ito}
\address{Research Institute for Mathematical Sciences, Kyoto university
Kyoto, 606-8502, Japan}
\email{tetitoh@kurims.kyoto-u.ac.jp}
\urladdr{http://www.kurims.kyoto-u.ac.jp/~tetitoh/}
\thanks{This research was partially supported by JSPS Research Grant-in-Aid for Research Activity Start-up. }

\begin{document}

\begin{abstract}
We show that a non-trivial, non-central normal subgroup of the braid groups contains a braid whose closure is a hyperbolic knot with arbitrary large genus. This shows that non-faithfulness of a quantum representation implies that the corresponding quantum invariant fails to detect the unknot. The proof utilizes the Dehornoy ordering of the braid groups.
\end{abstract}
\maketitle

\section{Introduction}

The problem whether the Jones polynomial and various other knot polynomials (HOMFLY, Kauffman, colored Jones, etc...) detect the unknot or not is one of the central open problem in knot theory.

These polynomial knot invariants are the typical and fundamental examples of  \emph{quantum invariants}, which are obtained in the following manner:
Let $V$ be a finite dimensional $U_{q}(\mathfrak{g})$-module, where $U_{q}(\mathfrak{g})$ denotes the quantum enveloping algebra of a complex semi-simple Lie algebra $\mathfrak{g}$. $U_{q}(\mathfrak{g})$ is a ribbon Hopf algebra so one has a linear representation of the braid group $\rho_V: B_n \rightarrow \textrm{GL}(V^{\otimes n})$. 
Let $K$ be an oriented knot in $S^{3}$ represented as a closure of an $n$-braid $\beta$. By taking a variant of trace of $\rho_V(\beta)$ called a quantum trace, we obtain an invariant $Q^{V}(K)$ of knot $K$, called a {\em quantum $V$-invariant}.
See \cite{oht} for details and basics of the quantum invariants.

As this construction suggests, a non-trivial kernel element of a quantum representation seems to produce a non-trivial knot with trivial quantum invariants. In \cite{big}, Bigelow pointed out this simple observation together with several reasonably-sounding conjectures, and proposed an approach toward the unknot detection problem of quantum invariants, in particular Jones polynomial, by analysing the faithfulness problem of quantum representations.

In this paper, we prove (a version of) Bigelow's conjecture in a slightly stronger form. 
\begin{theorem}
\label{theorem:main}
\cite[Conjecture 3.2]{big}
Let $H$ be a non-trivial, non-central normal subgroup of $B_{n}$ and let $\alpha$ be an $n$-braid whose closure is the unknot.
Then for any $N>0$, there exists $\beta \in K$ such that the closure of $\alpha\beta$ is a hyperbolic knot whose genus is greater than or equal to $N$ and whose braid index is $n$. 
\end{theorem}

As a corollary, we have the following implications of the non-faithfulness of a quantum representation.

\begin{corollary}
\label{cor:trivial}
Let $\rho_{V}:B_{n} \rightarrow \textrm{GL}(V^{\otimes n})$ be a quantum representation associated with a finite dimensional $U_{q}(\mathfrak{g})$-module $V$. If $\rho_{V}$ is not faithful for some $n$, then for any knot $K$, there exist infinitely many mutually different (hyperbolic) knots $K_{i}$ $(i=1,2,\ldots)$ such that the quantum $V$-invariant of $K_{i}$ and $K$ are the same for all $i$.
\end{corollary}

Thus if the quantum representation $\rho_{V}$ is non-faithful, then the corresponding knot invariant does not completely distinguish knots, and in particular, fails to detect the unknot. It deserves to point out the following stronger version of Corollary \ref{cor:trivial} which follows from the fact that the intersection of two non-central non-trivial normal subgroups of $B_{n}$ is always non-trivial \cite[Lemma 2.1]{lo}.

\begin{corollary}
\label{cor:trivial2}
Let $\rho_{V_j}:B_{n} \rightarrow \textrm{GL}(V_{j}^{\otimes n})$ $(j=1,2,\ldots,m)$ be a family of quantum representations. If all of $\rho_{V_j}$ are not faithful on $B_{n}$, then for any knot $K$, there exist infinitely many mutually different (hyperbolic) knots $K_{i}$ $(i=1,2,\ldots)$ such that all of the quantum $V_j$-invariant of $K_{i}$ and $K$ are the same.
\end{corollary}

Moreover, as is explained in \cite{big}, Theorem \ref{theorem:main} particularly shows the following, which adds an importance of one of the most famous open problem in the braid groups: Is the Burau representation for $B_{4}$ faithful ?

\begin{corollary}
If the Burau representation of the 4-strand braid group $B_{4}$ is not faithful, then the Jones polynomial fails to detect the unknot. 
\end{corollary}

Our proof of Theorem \ref{theorem:main} uses the Dehornoy ordering $<_{D}$, the standard left-ordering of the braid group $B_{n}$. 
See \cite{ddrw} for the definition and the basics of the Dehornoy ordering.

Let $\Delta$ be the braid that corresponds to the half-twist on $n$-strands, given by
\[ \Delta=(\sigma_{1}\sigma_{2}\cdots\sigma_{n-1})(\sigma_{1}\sigma_{2}\cdots \sigma_{n-2})\cdots(\sigma_{1}\sigma_{2})(\sigma_{1}). \]
$\Delta^{2}$ generates the central subgroup of $B_{n}$.
In \cite{i1,i2}, we have proved the following relationships between knots and the Dehornoy ordering.

\begin{theorem}
\label{theorem:i}
Let $K$ be a knot represented as a closure of an $n$-braid $\beta$.
\begin{enumerate}
\item \cite[Corollary 1.3]{i2} If $\beta >_{D} \Delta^{2N}$. Then $g(K)\geq N$.
\item \cite[Theorem 2.8]{i1} There exists $r(n) \in \Z$ such that if $\beta>_{D} \Delta^{2r(n)}$, then the braid index of $K$ is equal to $n$.
\item \cite[Theorem 1.3]{i1} If $\beta >_{D} \Delta^{4}$ and $\beta$ is a pseudo-Anosov braid, then $K$ is a hyperbolic knot.
\end{enumerate}
\end{theorem}

The proof of Theorem \ref{theorem:main} is a direct consequence of Theorem \ref{theorem:i} and the following property of the Dehornoy ordering.

\begin{theorem}
\label{theorem:unbounded}
If $H$ is a non-trivial normal subgroup of $B_{n}$, then $H$ is unbounded with respect to the Dehornoy ordering. Namely, for any $\alpha \in B_{n}$, there exists $\beta \in H$ such that $\alpha <_{D} \beta$. Moreover, if $H$ is non-central then one can take such $\beta$ as a pseudo-Anosov braid.
\end{theorem}

\begin{proof}[Proof of Theorem \ref{theorem:main}]
By Theorem \ref{theorem:unbounded}, for a non-trivial, non-central normal subgroup $H$ of $B_{n}$, there is a pseudo-Anosov pure braid $\beta \in H$ such that $\beta >_{D} \alpha^{-1} \Delta^{2N}$ for any integer $N \in \Z$. By a theorem of Papadopoulos \cite{pap}, by taking powers of $\beta$ if necessary, we may also assume that $\alpha\beta$ is a pseudo-Anosov.
Hence Theorem \ref{theorem:i} shows that the closure of $\alpha\beta$ is a hyperbolic knot with $g(K)\geq N$ and braid index is $n$.
\end{proof}

In the rest of the paper, we prove Theorem \ref{theorem:unbounded}.
Here we give an algebraic proof that uses the alternating decomposition \cite{deh,i0}.
Although this proof looks a bit strange and somewhat indirect, the alternating decomposition has an advantage that it gives an explicit and computational way to construct an arbitrary large braid in the normal subgroup $H$, starting from only one element in $H$.

Let $B_{n}^{+}$ be the set of positive braids, which are braids written as a product of positive standard generators $\{\sigma_{1},\ldots,\sigma_{n-1}\}$.
For $x,y \in B_{n}^{+}$, we define $x \succcurlyeq y$ if $xy^{-1} \in B_{n}^{+}$.
It is well-known that $\succcurlyeq$ is a lattice ordering of $B_{n}^{+}$. Let $M_{A}$ and $M_{B}$ be the submonoid of $B_{n}^{+}$ generated by $\{\sigma_{1},\ldots,\sigma_{n-2}\}$ and $\{\sigma_{2},\ldots,\sigma_{n-1}\}$, respectively. For $x \in B_{n}^{+}$ and $\ast=A,B$, we denote the $\succcurlyeq$-maximal element $y$ in $M_{\ast}$ that satisfies $x \succcurlyeq y$ by $x \wedge M_{\ast}$.

\begin{definition}
The \emph{alternating decomposition} of $x \in B_{n}^{+}$ is a decomposition of $x$ as a product of elements in $M_{A}$ and $M_{B}$ of the form
\[ \mathcal{A}(x) = B_{m}A_{m}\cdots B_{1}A_{1}B_{0} \]
where $A_i \in M_A$ and $B_i \in M_B$ are inductively defined by
\begin{gather*}
\begin{cases}
B_0 =  x \wedge M_{B},\\
A_{i} = x(B_{0}^{-1}A_{1}^{-1}\cdots B_{i-1}^{-1}) \wedge M_{A},\\
B_{i} = x(B_{0}^{-1}A_{1}^{-1}\cdots B_{i-1}^{-1}A_{i}^{-1}) \wedge M_{B}.
\end{cases}
\end{gather*}
We denote the length $m$ of the alternating decomposition by $\ell_{\mathcal{A}}(x)$, and call the \emph{alternating length} of $x$.
\end{definition}

For example, the alternating decomposition of $\Delta^{2N}$ $(N>0)$ is given by
\[ \mathcal{A}(\Delta^{2N})= (\sigma_{1})\underbrace{B'A'\cdots B'A'}_{N-1}B'(\sigma_{n-2}\cdots \sigma_{1})\Delta_{B}^{2N}\]
where $A'$, $B'$ and $\Delta_{B}$ are given by
\begin{gather}
\label{eqn:alt}
\begin{cases}
A' = (\sigma_{n-2}\cdots\sigma_{2}\sigma_{1}^{2}) \in M_{A}\\
B' = (\sigma_{2}\cdots \sigma_{n-2}\sigma_{n-1}^{2}) \in M_{B}\\
\Delta_{B}= (\sigma_{2}\sigma_{3}\cdots \sigma_{n-1})(\sigma_{2}\sigma_{3}\cdots \sigma_{n-2})\cdots(\sigma_{2}\sigma_{3})(\sigma_{2}) \in M_{B}.
\end{cases}
\end{gather}
For later use, we put 
\[ \Delta_{A} = \Delta \Delta_{B} \Delta^{-1} =(\sigma_{1}\sigma_{2}\cdots \sigma_{n-2})(\sigma_{1}\sigma_{2}\cdots \sigma_{n-1})\cdots(\sigma_{1}\sigma_{2})(\sigma_{1}) \in M_{A}. \]

The alternating decomposition nicely reflects the Dehornoy ordering. Here we summarize relationships between the alternating decomposition and the Dehornoy ordering which are needed to prove Theorem \ref{theorem:unbounded}. 

For $\beta \in B_{n}$, we will write $\beta \in B_{n-1} \subsetneq B_{n}$ if $\beta$ lies in the subgroup of $B_{n}$ generated by $\sigma_{1},\ldots,\sigma_{n-2}$ (which is, of course, isomorphic to $B_{n-1}$).

\begin{proposition}
\label{proposition:D-order}
For $x,x' \in B_{n}^{+}$, let 
\[ \mathcal{A}(x) = B_{m}A_{m}\cdots B_{1}A_{1}B_{0} \ \text{and} \ \mathcal{A}(x') = B'_{m'}A'_{m'}\cdots B'_{1}A'_{1}B'_{0} \]
be the alternating decompositions of $x$ and $x'$ respectively.
\begin{enumerate}
\item If $m=m'$ and $B_{m}\preccurlyeq B'_{m}$, then $x <_{D} x'$.
\item $\Delta^{2\ell_{\mathcal{A}}(x)-4}<_{D} x <_{D}
\Delta^{ 2\ell_{\mathcal{A}} (x)}$.
\item If $x <_{D}\Delta^{2\ell_{\mathcal{A}}(x)-2}$, then 
\[ \mathcal{A}(x)= (\sigma_{1})\underbrace{B'A'\cdots B'A'}_{N-1}B'(\sigma_{n-2}\cdots \sigma_{1})B_{0} \ \textrm{and} \ B_{0} <_{D} \Delta_{B}^{2\ell_{\mathcal{A}}(x)-2}, \]
where $A'$, $B'$ and $\Delta_{B}$ are as given in (\ref{eqn:alt}).
In particular, 
\[ \Delta (x^{-1}\Delta^{2 \ell_{\mathcal{A}} (x)-2}) \Delta^{-1} \in B_{n-1} \subsetneq B_{n}. \] 
\end{enumerate}
\end{proposition}

These results follow from the following facts.
Iterated use of the alternating decomposition defines a certain normal form of positive braids called the \emph{alternating normal form} or \emph{$\Phi$-normal form}. Using such normal forms, one defines a total ordering $<_{+}$ on $B_{n}^{+}$.
A Theorem of Dehornoy \cite[Proposition 5.19]{deh} based on Burckel's result \cite{bur} or Theorem of the author \cite[Theorem 1.1]{i0} shows that $<_{+}$ coincides with the Dehornoy ordering $<_{D}$. Therefore the alternating decomposition provides another combinatorial formulation of the Dehornoy ordering. (Actually, using a variant of the alternating decomposition, one obtains a combinatorial description of a finite Thurston type ordering, a geometric generalization of the Dehornoy ordering \cite{i0}.)

All the statements of Proposition \ref{proposition:D-order} are easily checked by identifying the Dehornoy ordering $<_{D}$ with the ordering $<_{+}$ based on the alternating decompositions. See \cite{deh,i0} or \cite[Chapter VII]{ddrw} for more precise relationships between the Dehornoy ordering and the alternating decompositions.

We also use the following fact. 
\begin{lemma}
\label{lemma:ind}
 If an $n$-braid $\beta \in B_{n-1}\subsetneq B_{n}$ satisfies $\Delta_{A}^{4} <_{D} \beta$, then for $N>0$, 
\[ \Delta^{2N} <_{D} \beta(\Delta \beta \Delta^{-1} \beta)^{N}\]
holds.
\end{lemma}

This assertion is easily checked by using a geometric (curve diagram) definition of the Dehornoy ordering given in \cite{fgrrw} (see \cite[Chapter X]{ddrw} as well), as we pictorially show in Figure \ref{fig:proof}.

\begin{figure}[htbp]
 \begin{center}
\includegraphics*[width=125mm]{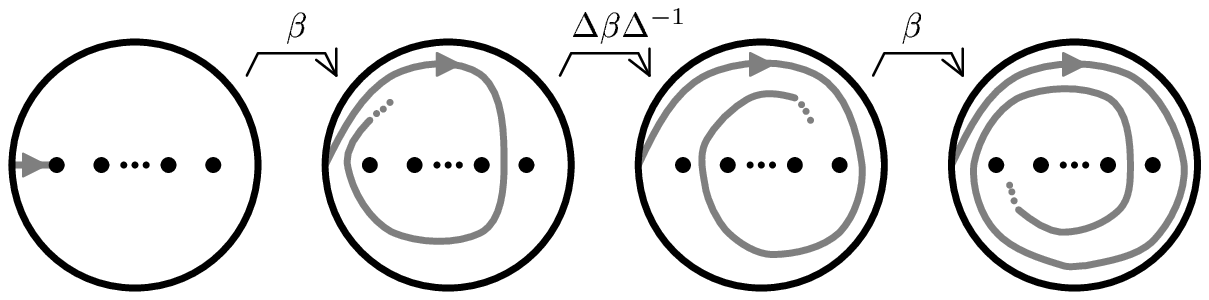}
\caption{Pictorial proof of Lemma \ref{lemma:ind}.}
\label{fig:proof}
\end{center}
\end{figure}
Now we are ready to prove Theorem \ref{theorem:unbounded}.

\begin{proof}[Proof of Theorem \ref{theorem:unbounded}]

First we show the unboundedness of $H$ by induction on $n$. The case $n=2$ is obvious, so we assume $n>2$.

Let $\gamma \in H$ be a non-trivial element of a normal subgroup $H$. 
By taking $\gamma^{-1}$ if necessary, we may assume that $\gamma <_{D} 1$.
Take $N>0$ so that $\Delta^{2N} \gamma \in B_{n}^{+}$ and let
$\mathcal{A}(\Delta^{2N} \gamma)=B_{m}A_{m}\cdots B_{1}A_{1}B_{0}$ be the alternating decomposition of $\Delta^{2N} \gamma$. 

Put $\gamma_{0}= B_{m}^{-1}\gamma B_{m} \in H$.
Then the alternating decomposition of $\Delta^{2N}\gamma_{0}$ is given by
\[ \mathcal{A}(\Delta^{2N} \gamma_{0})=A_{m}\cdots B_{1}A_{1}(B_{0}B_{m}).\]
By Proposition \ref{proposition:D-order} (1), $\Delta^{2N}\gamma_{0} <_{D} \Delta^{2N}\gamma$, so $\gamma_{0}<_{D} \gamma<_{D} 1$.
Moreover, by Proposition \ref{proposition:D-order} (2),
\[ \Delta^{2m-4} <_{D} \Delta^{2N} \gamma_{0} <_{D} \Delta^{2m}. \]
Since $\gamma_{0} <_{D} 1$, we have $m-N \leq 1$.

Assume that $m-N=1$. Then $\gamma_{0} <_{D} 1$ if and only if $A_{m}\cdots B_{1}A_{1}(B_{0}B_{m}) <_{D} \Delta^{2m-2}$. By Proposition \ref{proposition:D-order} (3), this implies that $\Delta^{-1}\gamma_{0}\Delta \in B_{n-1} \subsetneq B_{n}$.
Therefore by induction there exists $\gamma_{1} \in H$ such that 
$\gamma_{1} \in B_{n-1} \subsetneq B_{n}$ with $\Delta_{A}^{4} <_{D} \gamma_{1}$. Lemma \ref{lemma:ind} shows that $\gamma_{1} (\Delta \gamma_{1}\Delta^{-1}\gamma_{1})^{N} \in H$ satisfies
\[ \Delta^{2N} <_{D} \gamma_{1} (\Delta \gamma_{1}\Delta^{-1}\gamma_{1})^{N} \]
for any $N>0$ hence $H$ is unbounded.

Now we assume that $m-N \leq 0$. Let us put $\gamma_{+} = \Delta^{2N}\gamma_{0} \in B_{n}^{+}$ and for $i\geq 0$, let
\[ \gamma_{i}= [ \gamma_{+}(\Delta^{-1}\gamma_{+}\Delta)]^{i}\gamma_{+}. \]
Then $\ell_{\mathcal{A}}(\gamma_{i}) \leq (2i+1)m -i$, hence by Proposition \ref{proposition:D-order} (2) we conclude
\[ \Delta^{-2(2i+1)N}\gamma_{i} <_{D} \Delta^{2(2i+1)(m-N)-2i} <_{D} \Delta^{-2i}.\]
By definition $\Delta^{-2(2i+1)N}\gamma_{i} \in H$ hence we conclude that $H$ is unbounded.

It remains to show that one can find an arbitrary large pseudo-Anosov braid in $H$. Recall that every non-trivial non-central normal subgroup of $B_{n}$ contains a pseudo-Anosov braid $\beta_{\sf pA} >_{D} 1$ \cite[Lemma 2.5]{lo}: for a non-central element $\beta \in H$, $\theta^{N} \beta \theta^{-N} \beta^{-1} \in H$ is pseudo-Anosov if $N$ is sufficiently large and $\theta \in B_{n}$ is pseudo-Anosov which does not commute with $\beta$.

For a given $\alpha \in B_{n}$, take $\beta_{0}\in H$ so that $\beta_{0}>\alpha$. 
A theorem of Papadopolus \cite{pap} says that $\beta_{0}\beta_{\sf pA}^{N}$ is pseudo-Anosov for sufficiently large $N$. Since $\beta_{0}\beta_{\sf pA}^{N} >_{D}\beta_{0} >_{D} \alpha$, $\beta_{0}\beta_{\sf pA}^{N} \in H$ is a desired pseudo-Anosov element of $H$.
\end{proof}

\end{document}